\documentclass[11pt, reqno]{amsart}%

\usepackage{amsmath}
\usepackage{array}
\usepackage{amssymb}

\usepackage[margin=1in]{geometry}
\usepackage[bookmarks=false]{hyperref}
\usepackage{epstopdf}

\def\xx{{\bf x}}

\def\+{{\oplus}}
\newcommand{\Ff}{{\mathbb F}}

\newtheorem{thm}{Theorem}[section]
\newtheorem{theorem}{Theorem}[section]
\newtheorem{lemma}[thm]{Lemma}
\newtheorem{cor}[thm]{Corollary}

\newtheorem{prop}[thm]{Proposition}

\newtheorem{defn}[thm]{Definition}

\newtheorem{remark}[thm]{Remark}

\numberwithin{equation}{section}

\begin{document}

\title[Multistate nested canalizing functions]{Multistate nested canalizing functions}

\author[J O Adeyeye, C Kadelka, R Laubenbacher, Y Li]{ John O. Adeyeye $^{1\dag}$, Claus Kadelka$^{2\ddag\star}$, Reinhard Laubenbacher$^{3\ddag}$, Yuan Li$^{4\dag}$}
\address{{\small $^{1}$Department of Mathematics, Winston-Salem State University, NC
27110, USA}\\ \newline
{\small email: adeyeyej@wssu.edu }\\ \newline
{\small $^{2}$ Bioinformatics Institute \& Department of Mathematics, Virginia Tech, Blacksburg, VA
24061, USA }\\
{\small email: claus89@vt.edu}\\ \newline
{\small $^{3}$ Bioinformatics Institute \& Department of Mathematics, Virginia Tech, Blacksburg, VA
24061, USA}\\
{\small email: reinhard@vbi.vt.edu}\\ \newline
{\small $^{4}$ Department of Mathematics, Winston-Salem State University, NC 27110, USA}\\ \newline
{\small email: liyu@wssu.edu}}
\thanks{\newline $^{\dag}$ Supported by an award from the US DoD $\#$ W911NF-11-10166.\\
$^{\ddag}$ Supported by NSF Grant CMMI-0908201.\\
$^\star$ Corresponding author: claus89@vt.edu}
\keywords{nested canalizing function, layer number, indicator function, finite fields}
\date{}

\begin{abstract}
The concept of a \emph{nested canalizing} Boolean function has been studied over the course
of the last decade in the context of 
understanding the regulatory logic of molecular interaction networks, such as gene regulatory networks. Such functions
appear preferentially in published models of such networks. Recently, this concept has been generalized to 
include multi-state functions, and a recursive formula has been derived for their number, as a function of the number
of variables. This paper carries out a detailed analysis of the class of nested canalizing functions over an arbitrary
finite field. Furthermore, the paper generalizes the concept further, and derives a closed formula for the number
of such generalized functions. The paper also
derives a closed formula for the number of equivalence classes under permutation of variables. This is motivated by the
fact that two nested canalizing functions that differ by a permutation of the variables share many important properties
with each other. The paper contributes to the effort of identifying a class of functions over finite fields that are of 
interest in biology and also have interesting mathematical properties. 
\end{abstract}
\maketitle

\interdisplaylinepenalty=2500

\section{Introduction}

\label{sec-intro} Canalizing Boolean functions were introduced by S. Kauffman \cite{Kau1} as appropriate rules in 
Boolean network models of gene regulatory networks. In \cite{Win2}, a formula for the number of canalizing 
Boolean functions was presented. 
More recently, a subclass of these functions, so-called \emph{nested canalizing} functions (NCFs) was introduced \cite{Kau2}. 
The interest in studying these came from dynamic stability properties of networks constructed from such functions. 
Later, canalizing functions were generalized to any finite set, and a closed formula 
for the number of such functions was obtained \cite{Yua2}. 
A multi-state version of nested canalizing functions has been introduced in \cite{Mur,Mur2}, 
where it was shown that networks whose dynamics are controlled by multi-state nested canalizing functions have stability properties
similar to the Boolean case, namely large attractor basins and short limit cycles. 
An analysis of published Boolean and multi-state models of molecular regulatory networks revealed that the large majority of regulatory rules is canalizing, with most rules being in fact nested canalizing \cite{Har, Kau3, Nik, Mur}. 
Thus, nested canalizing rules and the properties of networks governed by them are important to study because 
of their relevance in systems biology. Furthermore, they also play a role in computational science; 
it was shown that the class of Boolean NCFs 
is identical to the class of so-called unate cascade Boolean functions, 
which has been studied extensively in engineering and computer science \cite{Abd2}. 
This class, in turn, corresponds exactly to the class of Boolean functions with corresponding 
binary decision diagrams of shortest average path length \cite{But}. 
Thus, a more detailed mathematical study of NCFs has applications to problems in engineering and computer science as well. For Boolean NCFs, a comprehensive analysis was given in \cite{Yua1}. 
For multi-state NCFs, based on the definition in \cite{Mur2}, a similar analysis is described in this paper. 
We provide a polynomial form of NCFs and an explicit formula of the cardinality of the number of such functions
for a given number of variables. The variables of any NCF are sorted into different layers, according to their dominance. 
Finally, we generalize the concept of NCF and study the number of equivalence classes obtained by permuting the variables.
The paper is arranged as follows. In Section \ref{2},  definitions and notations will be provided, which are then used in Section \ref{3} to prove the main results of the paper. Section \ref{4} deals with a generalization of the concept of nested canalizing function, and we finish with a conclusion in Section \ref{5}.

%%%%%%%%%%%%%%%%%%%%%%%%%%%%%%%%%%

\section{Definitions and Notation}\label{2} 
In this section we review some concepts and definitions from \cite{Mur, Mur2} 
to introduce the computational concept of \emph{canalization}. 
Let $\Ff=\Ff_{p}$ be a finite field with $p$ elements, where $p$ is prime. 
It is well-known \cite{Lid} that a function $f:\Ff^{n}\rightarrow \Ff$ can be expressed as a polynomial, 
called the algebraic normal form (ANF) of $f$:
\begin{displaymath}
f(x_{1},\ldots,x_{n})=\sum_{\substack{0\leq k_i\leq p-1,\\i=1,\ldots,n}}a_{k_{1}\ldots k_{n}}{x_{1}}^{k_{1}}\cdots{x_{n}}^{k_{n}},%
\end{displaymath}
with coefficients $a_{k_{1}\ldots k_{n}}\in\Ff$. The multivariate degree of the (nonzero) term $a_{k_{1}k_{2}\ldots k_{n}}{x_{1}}^{k_{1}}{x_{2}}^{k_{2}}\cdots {x_{n}}^{k_{n}}$ is defined as $k_{1}+k_{2}+\cdots+k_{n}$. 
Like in the one dimensional case, the greatest degree of all the terms of $f$ is called its {\it algebraic degree}, denoted by $\text{deg}(f)$.  

\begin{defn} \label{def2.1} 
A function $f(x_{1},x_{2},\ldots, x_{n})$ is essential in the variable $x_{i}$ if there exist $r, s\in\Ff$  such that 
\[f(x_{1},\ldots,x_{i-1},r,x_{i+1},\ldots,x_{n})\neq f(x_{1},\ldots,x_{i-1},s,x_{i+1},\ldots,x_{n}).\]
as functions.
\end{defn}

\begin{defn}\label{def2.2}\cite{Yua2} 
Given $a, b\in\Ff$ and $i\in\{1,\ldots,n\}$, a function $f(x_{1},x_{2},\ldots,x_{n})$ is $<i:a:b>$ canalizing if for all $x_j, j\neq i$ $$f(x_{1},\ldots,x_{i-1},a,x_{i+1},\ldots,x_{n})=b.$$
We call $x_i$ the canalizing variable of $f$. The set $S$ of all $a\in \Ff$ such that f is $<i:a:b>$ canalizing for some $b$
will be called the canalizing input set with respect to $x_i$, and $b$ will be called the canalized output with respect to $x_i$ and $a$.
\end{defn}

For a function $f$, we call $g$ a {\it subfunction} of $f$ if $g$ was obtained by fixing the values of some variables of $f$. 
The proof of the following proposition is straightforward.

\begin{prop}\label{prop*}
If a function $f(x_{1},x_{2},\ldots,x_{n})$ is both $<i_1:a_1:b_1>$ and $<i_2:a_2:b_2>$ canalizing, then $b_1=b_2$. If $x_i$ is a canalizing variable of $f$ which is still essential in a subfunction $g$, then $x_i$ is also a canalizing variable of $g$.
\end{prop}

We now assume that $\Ff=\{0,1,\ldots,p-1\}$ is ordered, in the natural order $0<1<\cdots <p-1$. 
A proper subset $S$ of $\Ff$ is called an interval if and only if $S=\{0,\ldots,j\}$ or $S^c=\Ff-S=\{0,\ldots,j\}$, where $0\leq j<p-1$. Hence, a proper subset $S$ is an interval if and only if $S^c$ is an interval. 

\begin{defn}\label{def2.3}\cite{Mur2}
Let $f$ be a  function in $n$ variables and $S_i$ be intervals of $\Ff$, $i=1,\ldots,n$. Let $\sigma$ be a permutation of the set $\{1,2,\ldots,n\}$. Then $f$ is a \emph{nested canalizing function (NCF)} in the variable order $x_{\sigma(1)},\ldots,x_{\sigma(n)}$ with canalizing input sets $S_{1},\ldots,S_{n}$ and canalized output values $b_{1},\ldots,b_{n},b_{n+1}$ with $b_n\neq b_{n+1}$, if it can be represented in the form
\[f(x_{1},\ldots,x_{n})=
\left\{\begin{array}[c]{ll}
b_{1} & x_{\sigma(1)}\in S_{1},\\
b_{2} & x_{\sigma(1)} \notin{ S_{1}}, x_{\sigma(2)}\in S_{2},\\
b_{3} & x_{\sigma(1)}\notin{ S_{1}}, x_{\sigma(2)} \notin{ S_{2}}, x_{\sigma(3)}\in S_{3},\\
\vdots  & \\
b_{n} & x_{\sigma(1)} \notin{ S_{1}},\ldots,x_{\sigma(n-1)} \notin{ S_{n-1}}, x_{\sigma(n)}\in S_{n},\\
{b_{n+1}} & x_{\sigma(1)} \notin{ S_{1}},\ldots,x_{\sigma(n-1)} \notin{ S_{n-1}}, x_{\sigma(n)}\notin{S_{n}}.
\end{array}\right.\]

In short, the function $f$ is said to be nested canalizing if $f$ is nested canalizing in some variable order with some canalizing input sets and some canalized output values.
\end{defn}

Let $\mathbb{S}=(S_{1},S_{2},\ldots,S_{n})$ and $\beta=(b_{1},b_{2},\ldots,b_{n+1})$ with $b_n\neq b_{n+1}$. We say that $f$ is $\{\sigma:\mathbb{S}:\beta\}$ NCF if it is nested canalizing in the variable order $x_{\sigma(1)},\ldots,x_{\sigma(n)}$, 
with canalizing input sets $\mathbb{S}=(S_{1},\ldots,S_{n})$ and canalized output values $\beta=(b_{1},\ldots, b_{n+1})$.

This definition immediately implies the following result. 

\begin{prop}\label{prop2.1}
A function $f$ is $\{\sigma:\mathbb{S}:\beta\}$ NCF  if and only if  $f$ is
$\{\sigma:\mathbb{S'}:\beta'\}$ NCF, where $\mathbb{S'}=(S_{1},S_{2},\ldots,{S_n}^c)$ and $\beta'=(b_{1},b_{2},\ldots,b_{n-2},b_{n+1},b_n)$.
\end{prop}

\begin{prop}\label{prop2.2} 
Let $f(x_{1},\ldots,x_{n})$ be $\{\sigma:\mathbb{S}:\beta\}$ NCF, i.e., $f$ is NCF in the variable order $x_{\sigma(1)},\ldots,x_{\sigma(n)}$, with canalizing input sets $\mathbb{S}=(S_1,\ldots,S_{n})$ and canalized output values $\beta=(b_{1},\ldots, b_{n+1})$. For $1\leq k\leq n-1$, fix $x_{\sigma(1)}={a_{1}}, \ldots,x_{\sigma(k)}={a_{k}}$, where $a_i\in {S_i}^c$.\\
Then the subfunction $f(x_{1},\ldots,\overset{\sigma(1)}{{a_{1}}},\ldots,\overset{\sigma(k)}{{a_{k}}},\ldots, x_{n})$ is $\{\sigma^{*}:\mathbb{S}^{*}:\beta^{*}\}$ NCF on the remaining variables, where $\sigma^{*}=(\sigma(k+1),\ldots,\sigma(n))$, ${\mathbb{S}}^{*}=(S_{k+1},\ldots,S_{n})$, 
and $\beta^{*}=(b_{k+1},\ldots,b_{n+1})$.
\end{prop}

Note that all variables appearing in the definition of an NCF must be essential. A constant function $f=b$ is however $<i:a:b>$ canalyzing for any $i$ and $a$. Thus, when a function is said to be nested canalizing, we assume that this function is NCF in its essential variables.

%%%%%%%%%%%%%%%%%%%%%%%%%%%%%%%%%%%%%%

\section{Characterization of  Nested Canalizing Functions}\label{3}
In the Boolean case, the extended monomial plays an important role in determining a novel polynomial form of NCF's \cite{Yua1}. In the multistate case, the product of indicator functions, which was used in \cite{Mur2}, will take over this role. 

\begin{defn}\label{def3.1}
Given a proper subset $S$ of $\Ff$,  the indicator function (of $S^c$) is defined as
\[Q_S(x)=\left\{\begin{array}[c]{ll}
0 & x\in S,\\
1 & x\in S^c.\end{array}\right.\]
\end{defn}

\begin{lemma}\label{lm3.1}
Let $a, b$ be any nonzero elements of $\Ff$, and let $S$ be any interval of $\Ff$. The number of different functions $f=bQ_S(x)+a$, which cannot be written as $cQ_{S'}(x)$, where $c\neq 0$ and $S'$ is an interval of $\Ff$, is
$(p-1)^2(p-2)$.
\end{lemma}
\begin{proof}
If a function $f(x)=bQ_S(x)+a$ can be written as $cQ_{S'}(x)$, then
\[f=bQ_S(x)+a=\left\{\begin{array}[c]{ll}%
a & x\in S\\
a+b & x\in S^c\end{array}\right.
=\left\{\begin{array}[c]{ll}%
0 & x\in S'\\
c & x\in {S'}^c\end{array}\right.
=cQ_{S'}(x).\]
Since $a$ and $c$ are nonzero, then $a+b = 0$ and, therefore, $a = -b$ must hold for such a function. 
Since $\Ff$ contains $p-1$ nonzero numbers, there are $p-1$ choices for $b$ and $p-2$ choices for $a$ to obtain $f\neq cQ_{S'}(x)$. 
Generally, there are $2(p-1)$ different intervals $S$ but only half of them lead to a different $f$ 
since every $f$ can be expressed in two different ways:
\[bQ_S(x)+a=b(1-Q_{S^c}(x))+a=-bQ_{S^c}(x)+(a+b).\]
Thus, there are $(p-1)(p-1)(p-2)=(p-1)^2(p-2)$ different functions $f=bQ_S(x)+a$ that cannot be written as $cQ_{S'}(x)$.
\end{proof}

From the above proof, it is clear that $bQ_S(x)+a=cQ_{S'}(x)$ for some $c\neq 0$ and $S'$ if and only if $a+b=0$.

\begin{lemma}\label{lm3.2}
Given $a,b \neq 0$ and intervals $S_i$, $i=1,\ldots,k$ with $k\geq 2$, then
\begin{enumerate}
\item $f(\xx)=f(x_1,\ldots,x_k)=b\prod_{j=1}^k Q_{S_j}(x_j)+a$ cannot be written as $c\prod_{j=1}^k Q_{S_j'}(x_j)$, where $c\neq 0$ and all $S_j'$ are intervals, $j=1,\ldots, k$.
\item There are $2^{k}(p-1)^{k+2}$ different functions of the form $b\prod_{j=1}^k Q_{S_j}(x_j)+a$.
\end{enumerate}
\end{lemma}

\begin{proof}
(1) As in the previous lemma, if a function $f(\xx)=b\prod_{j=1}^k Q_{S_j}(x_j)+a$ can be written as $c\prod_{j=1}^k Q_{S_j'}(x_j)$, then
\begin{align*}
b\prod_{j=1}^k Q_{S_j}(x_j)+a &= \left\{\begin{array}[c]{ll}%
a & \exists j: x_j\in S_j\\
a+b & \forall j: x_j \in S_j^c \end{array}\right. \\
&= \left\{\begin{array}[c]{ll}%
a & \xx \in (S_1 \times \Ff^{k-1}) \cup \ldots \cup (\Ff^{k-1} \times S_k)\\
a+b & \xx \in S_1^c \times \ldots \times S_k^c \end{array}\right. \\
&=\left\{\begin{array}[c]{ll}%
0 & \xx \in ({S_1'} \times \Ff^{k-1}) \cup \ldots \cup (\Ff^{k-1} \times {S_k'})\\
c & \xx \in {S_1'}^c \times \ldots \times {S_k'}^c \end{array}\right.\\
&=c\prod_{j=1}^k Q_{S_j'}(x_j).
\end{align*}
Since $a,c\neq 0$, $a+b=0$ must hold. Hence, 
\[S_1^c \times \cdots \times S_k^c = ({S_1'} \times \Ff^{k-1}) \cup \cdots \cup (\Ff^{k-1} \times {S_k'}).\]
This last statement is, however, impossible. Thus, there is no function $f(\xx)=b\prod_{j=1}^k Q_{S_j}(x_j)+a$ that can be written as $c\prod_{j=1}^k Q_{S_j'}(x_j)$.

(2) The nonzero constants $a,b$ can be arbitrarily chosen, with $p-1$ choices each. Contrary to the previous lemma, each choice of intervals $S_1, \ldots, S_k$ leads to a different function because $S_1^c \times \ldots \times S_k^c \neq ({S_1} \times \Ff^{k-1}) \cup \ldots \cup (\Ff^{k-1} \times {S_k})$ and because $a \neq a+b$. For each interval there are $2(p-1)$ choices, which is why altogether there are $(p-1)(p-1)(2(p-1))^k = 2^k(p-1)^{k+2}$ different functions of the form $b\prod_{j=1}^k Q_{S_j}(x_j)+a$.
\end{proof}

The following theorem states the main result of this section. It gives an algebraic characterization of nested canalizing functions.

\begin{theorem}\label{th2}
\label{th3.2} Given $n\geq2$, the function $f(x_{1},\ldots,x_{n})$ is nested canalizing if and only if  it can be uniquely written as
\begin{equation}\label{eq3.1}
f(x_{1},\ldots,x_{n})=M_{1}(M_{2}(\cdots(M_{r-1}(B_{r+1}M_{r}+B_r )+B_{r-1})\cdots)+B_2)+B_1,
\end{equation}
where each $M_{i}$ is a product of indicator functions of a set of disjoint variables. More precisely, $M_{i}=\prod_{j=1}^{k_{i}}(Q_{S_{i_j}}(x_{i_j})), i=1,\ldots,r, k_{i}\geq1$ for $i=1,\ldots,r, k_{1} + \cdots + k_{r}=n, B_2,\ldots,B_{r+1}\neq 0, B_1\in \Ff,\{i_{j} \big| j=1,\ldots,k_{i}, i=1,\ldots,r\}=\{1,\ldots,n\}$, and, if $k_r=1, then B_{r+1}+B_r\neq 0$.
\end{theorem}

\begin{proof}
First, let $b_i = \sum_{j=1}^i B_j$. Then it is straightforward to check that any function written as in Equation \ref{eq3.1} 
is a $\{\sigma':\mathbb{S'}:\beta'\}$ NCF, where
$\sigma'(x_1,\ldots,x_n)=(x_{1_1},\ldots,x_{1_{k_1}},\ldots,x_{r_1},\ldots,x_{r_{k_r}})$, $\mathbb{S'}=(S_{1_1},\ldots,S_{1_{k_1}},\ldots,S_{r_1},\ldots,S_{r_{k_r}})$, and
$\beta'=(\underbrace{b_1,\ldots,b_1}_{k_1},\underbrace{b_2,\ldots,b_2}_{k_2},\ldots,\underbrace{b_r,\ldots,b_r}_{k_r},b_{r+1})$.

Second, suppose $f$ is a $\{\sigma:\mathbb{S}:\beta\}$ NCF, where  $\mathbb{S}=(S_{1},S_{2},\ldots,S_{n})$ and $\beta=(b_{1},b_{2},\ldots,b_{n+1})$ with $b_n\neq b_{n+1}$. Then there exist $k_i,i=1,\ldots,r$, $k_1+\cdots+k_r=n$, $k_i\geq 1$, such that $b_1=\cdots=b_{k_1}=:C_1, b_{k_1+1}=\cdots=b_{k_1+k_2}=:C_2, \ldots, b_{k_1+\cdots+k_{r-1}+1}=\cdots=b_n=:C_r, b_{n+1}=:C_{r+1}$, 
and $C_j \neq C_{j+1}, j=1,\ldots, r$.
Let $B_1:=C_1, B_2:=C_2-C_1, \ldots, B_{r+1}=C_{r+1}-C_r$. Then $B_1\in \Ff, B_2,\ldots,B_{r+1} \in \Ff - \{0\}$. 
It is straightforward to check that $M_{1}(M_{2}(\cdots(M_{r-1}(B_{r+1}M_{r}+B_r )+B_{r-1})\cdots)+B_2)+B_1$, 
with these constants $B_1, \ldots, B_{r+1}$, now equals $f$, which shows that any NCF can be written as in Equation \ref{eq3.1}.

Finally, we need to show that each NCF has a unique polynomial representation, as indicated. 
Without loss of generality, let $\sigma$ be the identity permutation, i.e., let $f$ be nested canalizing in the variable order $x_1,\ldots,x_n$. Besides, let $f$ be of the form in Equation \ref{eq3.1}. Then all the variables of $M_1$, $x_1,\ldots,x_{k_1}$, are canalizing variables of $f$ with common canalized output $B_1$. We will now show that $x_1,\ldots,x_{k_1}$ are all the canalizing variables of $f$ to prove the uniqueness of $M_1$ and $B_1$.
If $x_1\in S_1^c,\ldots,x_{k_1}\in S_{k_1}^c$, then all the variables of $M_2$, $x_{k_1+1},\ldots,x_{k_2}$, are canalizing variables of the subfunction $f_1=M_{2}(\cdots(M_{r-1}(B_{r+1}M_{r}+B_r )+B_{r-1})\cdots)+(B_2+B_1)$. Since $B_1\neq B_1+B_2$, by Proposition \ref{prop*}, $x_{k_1+1},\ldots,x_{k_2}$ are not canalizing variables of $f$. In the same manner, all the variables of $M_3$ are not canalizing variables of $f_1$ and thus not canalizing variables of $f$ either, since a subfunction of a subfunction is also a subfunction of the original function. 
Iteratively, we can prove that $x_1,\ldots,x_{k_1}$ are the only canalizing variables of $f$, which makes $M_1$ and $B_1$ unique. In the same way, the uniqueness of $M_2, \ldots, M_r$ and $B_2, \ldots, B_{r+1}$ can be shown.\end{proof}

Because each NCF can be uniquely written in the form of Equation \ref{eq3.1}, 
the number $r$ is uniquely determined by $f$, and can be used to specify the class of nested canalizing functions. 

\begin{defn}\label{def3.3} 
For an NCF $f$, written in the form of Equation \ref{eq3.1}, let the number $r$ be called its \emph{layer number}. 
Essential variables of $M_{1}$ are called most dominant variables (canalizing variables), and are part of the first layer of $f$. Essential variables of $M_{2}$ are called second most dominant variables, and are part of the second layer. Iteratively, essential variables of $M_{k}$ are called $k$ most dominant variables, and are part of the $k^{\text{th}}$ layer.
\end{defn}

\begin{remark}\label{remark2}
In the Boolean case ($p=2$), each $M_i$ in Theorem \ref{th2} is an extended monomial, and, since $B_{r+1}+B_r=1+1=0$, $k_r$ is greater than 1. Thus, Theorem \ref{th2} reduces to its Boolean version, already stated as Theorem 4.2 in \cite{Yua1}. On the other hand, if $B_{r+1}+B_r$ would be zero, then $k_r=1$ is impossible since the function could be uniquely written in $r-1$ layers by the comments of Lemma \ref{lm3.1}. 
\end{remark}

A function can be nested canalizing in different variable orders but only all variables in the same layer can be reordered. More precisely, we have the following result.

\begin{cor}\label{CoA}
Let $\sigma$ and $\eta$ be two permutations of $\{1,\ldots,n\}$, and let $f$ be both $\{\sigma: \mathbb{S}: \beta\}$ NCF and $\{\eta: \mathbb{S}': \beta'\}$ NCF, then
\begin{enumerate}
\item \hspace{0.7in}$\{x_{\sigma(1)},\ldots, x_{\sigma(k_1)}\}=\{x_{\eta(1)},\ldots, x_{\eta(k_1)}\}$\\
\ \hspace{0.7in} $\vdots$\\
$\{x_{\sigma(k_1+\cdots+k_{r-1}+1)},\ldots, x_{\sigma(n)}\}=\{x_{\eta(k_1+\cdots+k_{r-1}+1)},\ldots, x_{\eta(n)}\}.$
\item $\beta=\beta'=(b_1,\ldots,b_n, b_{n+1})$ \text{and}\  $b_1=\cdots =b_{k_1}, \ldots , b_{k_1+\cdots+k_{r-1}+1}=\cdots =b_n$
\item $b_1\neq b_{k_1+1}, b_{k_1+1} \neq b_{k_1+k_2+1}, \ldots b_{k_1+\cdots +k_{r+1}+1}\neq b_{n+1}.$
\end{enumerate}
\end{cor}

\begin{proof}
By Theorem \ref{th3.2}, $f$ can be written as a unique polynomial of the form of Equation \ref{eq3.1}, 
and all equalities follow from this representation.
\end{proof}

This corollary allows to determine the layer number of any NCF by only looking at its canalizing output values. For example, if $p=3$ and if $f$ is nested canalizing with canalized output values $\beta = (1,1,1,0,0,0,2,0,0,2,2,1)$ ($n=11$), then the layer number of $f$ is $5$.

Let $\mathbb{NCF}(n)$ denote the set of all nested canalizing functions in $n$ variables.

\begin{cor}
\label{co3.1} 
For $n\geq2$, the number of nested canalizing functions is given by
\begin{align*}
|\mathbb{NCF}(n)|&=2^{n-1}p(p-2)\sum_{\substack{r=2}}^{n}(p-1)^{n+r-1}\sum_{\substack{k_{1}+\cdots+k_{r-1}=n-1\\k_{i}\geq1,i=1,\ldots,r-1}}\frac{n!}{k_{1}!k_{2}!\cdots k_{r-1}!}\\
&\: +2^{n}p\sum_{\substack{r=1}}^{n-1}(p-1)^{n+r}\sum_{\substack{k_{1}+\cdots+k_{r}=n\\k_{i}\geq1,i=1,\ldots,r-1,k_{r}\geq2}}\frac{n!}{k_{1}!k_{2}!\cdots k_{r}!}
\end{align*}
\end{cor}

\begin{proof}
If $r=1$, then $f=B_2M_1+B_1$. Similar to Lemma \ref{lm3.2}, the number of such functions is $(2(p-1))^n(p-1)p=2^n(p-1)^{n+1}p$, 
since $B_1 \in \Ff$ can be arbitrarily chosen, and $k_1=n$.

For $r>1$, Equation \ref{eq3.1} yields that for each choice of $k_{1},\ldots, k_{r}$, $k_{i}\geq1$, $i=1,\ldots,r$, there are $(2(p-1))^{k_{j}}\binom{n-k_{1}-\cdots-k_{j-1}}{k_{j}}$ ways to form $M_{j}$, $j=1,\ldots,r$. For those NCF's with $k_r=1$, by Lemma \ref{lm3.1}, there are $(p-1)^2(p-2)$ different functions of the form $B_{r+1}M_r+B_r$ with $B_r\neq 0$. For the remaining NCF's, i.e. those with $k_r>1$, Lemma \ref{lm3.2} yields that there are $(p-1)^2(2(p-1))^{k_r}$ ways to form $B_{r+1}M_r+B_r$, with $B_r\neq 0$.

Note that there are $p-1$ choices for each $B_i$, $2\leq r\leq B_{r-1}$, $p$ choices for $B_1$, and $2(p-1)$ choices for each canalizing input interval. Hence, the total number of NCF's with $r>1, k_r=1$, can be given by
{\small
\begin{align*}
N_1&=\sum_{r=2}^n\sum_{\substack{k_{1}+\cdots+k_{r-1}=n-1\\k_{i}\geq1,i=1,\ldots,r-1}}(2(p-1))^{k_{1}+\cdots+k_{r}}\binom{n}{k_{1}}\binom{n-k_{1}}{k_{2}}\cdots\binom{n-k_{1}-\cdots-k_{r-2}}{k_{r-1}}(p-1)^2(p-2)(p-1)^{r-2}p\\
&=2^{n-1}p(p-2)\sum_{r=2}^n\sum_{\substack{k_{1}+\cdots+k_{r-1}=n-1\\k_{i} \geq1,i=1,\ldots,r-1}}(p-1)^{n+r-1}\frac{n!}{(k_{1})!(n-k_{1})!}\frac{(n-k_{1})!}{(k_{2})!(n-k_{1}-k_{2})!}\frac{(n-k_{1}-\cdots-k_{r-2})!}{k_{r-1}!(n-k_{1}-\cdots-k_{r-1})!}\\
&=2^{n-1}p(p-2)\sum_{r=2}^n(p-1)^{n+r-1}\sum_{\substack{k_{1}+\cdots+k_{r-1}=n-1\\ k_{i}\geq1,i=1,\ldots,r-1}}\frac{n!}{k_{1}!k_{2}!\cdots k_{r-1}!}
\end{align*}}
Similarly, the total number of NCF's with $r>1, k_r>1$ is
%{\small \[N_2=2^{n}p\sum_{r=2}^{n-1}(p-1)^{n+r}\sum_{\substack{k_{1}+\ldots+k_{r}=n\\k_{i}\geq1,i=1,\ldots,r-1,k_r\geq 2}}\frac{n!}{k_{1}!k_{2}!\cdots k_{r}!}.\]}

{\small \begin{align*}
N_2&=\sum_{r=2}^{n-1}\sum_{\substack{k_{1}+\cdots+k_{r}=n\\k_{i}\geq1,i=1,\ldots,r-1,k_r\geq 2}}(2(p-1))^{k_{1}+\cdots+k_{r}}\binom{n}{k_{1}}\binom{n-k_{1}}{k_{2}}\cdots\binom{n-k_{1}-\cdots-k_{r-1}}{k_{r}}(p-1)^2(p-1)^{r-2}p\\
&=2^{n}p\sum_{r=2}^{n-1}\sum_{\substack{k_{1}+\cdots+k_{r}=n\\k_{i}\geq1,i=1,\ldots,r-1,k_r\geq 2}}(p-1)^{n+r}\frac{n!}{(k_{1})!(n-k_{1})!}\frac{(n-k_{1})!}{(k_{2})!(n-k_{1}-k_{2})!}\cdots \frac{(n-k_{1}-\cdots-k_{r-1})!}{k_{r}!(n-k_{1}-\cdots-k_{r})!}\\
&=2^{n}p\sum_{r=2}^{n-1}(p-1)^{n+r}\sum_{\substack{k_{1}+\cdots+k_{r}=n\\k_{i}\geq1,i=1,\ldots,r-1,k_r\geq 2}}\frac{n!}{k_{1}!k_{2}!\cdots k_{r}!}
\end{align*}}

By combining all three groups of NCF's, one gets that the total number of NCF's in $n$ variables is given by
\begin{equation}\label{eq3.2}
\begin{aligned}
|\mathbb{NCF}(n)|&=2^n(p-1)^{n+1}p+N_1+N_2\\
&=2^{n-1}p(p-2)\sum_{\substack{r=2}}^{n}(p-1)^{n+r-1}\sum_{\substack{k_{1}+\cdots+k_{r-1}=n-1\\k_{i}\geq1,i=1,\ldots,r-1}}\frac{n!}{k_{1}!k_{2}!\cdots k_{r-1}!}\\
&\: +2^{n}p\sum_{\substack{r=1}}^{n-1}(p-1)^{n+r}\sum_{\substack{k_{1}+\cdots+k_{r}=n\\k_{i}\geq1,i=1,\ldots,r-1,k_{r}\geq2}}\frac{n!}{k_{1}!k_{2}!\cdots k_{r}!}
%&=2^{n-1}p(p-2)\sum_{\substack{r=1}}^{n-1}(p-1)^{n+r}\sum_{\substack{k_{1}+\ldots+k_{r}=n\\k_{i}\geq1,i=1,\ldots,r-1, k_r=1 }}\frac{n!}{k_{1}!k_{2}!\ldots k_{r}!}\\
%&\: +2^{n}p\sum_{\substack{r=1}}^{n-1}(p-1)^{n+r}\sum_{\substack{k_{1}+\ldots+k_{r}=n\\k_{i}\geq1,i=1,\ldots,r-1,k_{r}\geq2}}\frac{n!}{k_{1}!k_{2}!\ldots k_{r}!}
\end{aligned}\end{equation}
\end{proof}

Note that for $p=2$ we get the same formula as in \cite{Yua1}. However, we are now also able to compute the number of multistate NCF's. For example, when $p=3$ and $n=2,3,4$, we get 192, 5568, 219468, respectively; when $p=5$ and $n=2,3,4$, we get 5120, 547840, 78561280, respectively. These results are consistent with those in \cite{Mur2}.

It has been shown in \cite{Mur2} that the number of multistate nested canalizing functions can be calculated recursively. Thus, by equating \ref{eq3.2} to the recursive relation, we have

\begin{cor}
\label{co3.2} For the nonlinear recursive sequence
\[a_{2}=4(p-1)^4, a_{n}=\sum_{r=2}^{n-1}\binom{n}{r-1}2^{r-1}(p-1)^ra_{n-r+1}+2^{n-1}(p-1)^{n+1}(2+n(p-2)) , n\geq3\]
it holds that 
\[\big|\mathbb{NCF}(n)\big| = pa_n,\]
and the explicit solution for $a_n$ is given by
\begin{align*}
a_n&=2^{n-1}(p-2)\sum_{\substack{r=2}}^{n}(p-1)^{n+r-1}\sum_{\substack{k_{1}+\cdots+k_{r-1}=n-1\\k_{i}\geq1,i=1,\ldots,r-1}}\frac{n!}{k_{1}!k_{2}!\cdots k_{r-1}!}\\
&+2^{n}\sum_{\substack{r=1}}^{n-1}(p-1)^{n+r}\sum_{\substack{k_{1}+\cdots+k_{r}=n\\k_{i}\geq1,i=1,\ldots,r-1,k_{r}\geq2}}\frac{n!}{k_{1}!k_{2}!\cdots k_{r}!}
\end{align*}
\end{cor}

%%%%%%%%%%%%%%%%%%%%%%%%%%%%%%%%

\section{Generalization and Permutation Equivalence of NCF's}\label{4}
In this section, we generalize the concept of \emph{nested canalizing functions} to any finite field $\Ff_q$, where $q$ is a power of a prime.
Besides, the canalizing sets  are no longer restricted to intervals containing exactly one endpoint but can be any proper subset of $\Ff_q$. Within this setting, similar results as in the previous section are obtained. Since the proofs are basically the same as in Section \ref{3}, we just list the main results and omit the proof. 

\begin{defn}\label{def4.1}
Let $f$ be a function in $n$ variables over $\Ff_q$, and let $S_i$ be proper subsets of $\Ff_q$, $i=1,\ldots,n$.  Let $\sigma$ be a permutation of the set $\{1,2,\ldots,n\}$. Then the function $f$ is a \emph{nested canalizing function} in the variable order $x_{\sigma(1)},\ldots,x_{\sigma(n)}$ with canalizing input sets $S_{1},\ldots,S_{n}$ and canalized output values $b_{1},\ldots,b_{n},b_{n+1}$, with $b_n\neq b_{n+1}$, if it can be represented in the form
\[f(x_{1},\ldots,x_{n})=\left\{\begin{array}[c]{ll}
b_{1} & x_{\sigma(1)}\in S_{1},\\
b_{2} & x_{\sigma(1)} \notin{ S_{1}}, x_{\sigma(2)}\in S_{2},\\
b_{3} & x_{\sigma(1)}\notin{ S_{1}}, x_{\sigma(2)} \notin{ S_{2}},x_{\sigma(3)}\in S_{3},\\
\vdots  & \\
b_{n} & x_{\sigma(1)} \notin{ S_{1}},\ldots,x_{\sigma(n-1)} \notin{ S_{n-1}}, x_{\sigma(n)}\in S_{n},\\
{b_{n+1}} & x_{\sigma(1)} \notin{ S_{1}},\ldots,x_{\sigma(n-1)} \notin{ S_{n-1}}, x_{\sigma(n)}\notin{S_{n}}.\end{array}\right.\]
The function $f$ is called nested canalizing if $f$ is nested canalizing in some variable order with some canalizing input sets and some canalized output values.
\end{defn}

This definition generalizes the concept of nested canalizing functions as defined in \cite{Mur}. 
However, NCF's defined in this way can still be uniquely represented as a polynomial as discussed above. 
The following theorem is a generalized version of Theorem \ref{th3.2} with the same arguments used to prove it.

\begin{theorem}\label{th4.1}
Given $n\geq2$, the function $f: \Ff_q^n \rightarrow \Ff_q$ is nested canalizing if and only if  it can be uniquely written as
\begin{equation}\label{eq4.1}
f(x_{1},\ldots,x_{n})=M_{1}(M_{2}(\cdots(M_{r-1}(B_{r+1}M_{r}+B_r )+B_{r-1})\cdots)+B_2)+B_1,
\end{equation}
where each $M_{i}$ is a product of indicator functions of a set of disjoint variables. More precisely, 
we have that $M_{i}=\prod_{j=1}^{k_{i}}(Q_{S_{i_j}}(x_{i_j})), i=1,\ldots,r$, $k_{i}\geq1$ for $i=1,\ldots,r, k_{1}+ \ldots + k_{r}=n$, 
$B_2,\ldots,B_{r+1}\neq 0$, $B_1\in \Ff$, $\{i_{j}|j=1,\ldots,k_{i}, i=1,\ldots,r\}=\{1,\ldots,n\}$, 
each $S_i$, $i=1,\ldots,n$ is any proper subset of $\Ff_q$, and, if $k_r=1$, then $B_{r+1}+B_r\neq 0$.
\end{theorem}

As in the previous section, we can count the number of NCF's and get
\begin{cor}
\label{co4.1} 
For $n\geq2$, the number of NCF's over $\Ff_q$, where NCF's are defined as in Definition \ref{def4.1}, is given by
\begin{align*}
&2^{n-1}q(q-2)(2^{q-1}-1)^{n} \sum_{\substack{r=2}}^{n}(q-1)^{r-1}\sum_{\substack{k_{1}+\cdots+k_{r-1}=n-1\\k_{i}\geq1,i=1,\ldots,r-1}}\frac{n!}{k_{1}!k_{2}!\cdots k_{r-1}!}\\
&+2^nq(2^{q-1}-1)^n\sum_{\substack{r=1}}^{n-1}(q-1)^{r}\sum_{\substack{k_{1}+\cdots+k_{r}=n\\k_{i}\geq1,i=1,\ldots,r-1,k_{r}\geq2}}\frac{n!}{k_{1}!k_{2}!\cdots k_{r}!}
\end{align*}
\end{cor}
\begin{proof}
The proof equals the one of Corollary \ref{co3.1}. We just replace the number of elements $p$ in the finite field by $q$ , the number of nonzero elements, $p-1$ by $q-1$, and the number of intervals, $2p-2$, by the number of proper subsets $2^q-2$, respectively.
\end{proof}
\begin{defn}\label{def4.2}
Given two functions $f(x_1,\ldots,x_n)$ and $g(x_1,\ldots,x_n)$ over $\Ff_q$. We call $f$ and $g$ permutation equivalent if there exists a permutation $\sigma$ such that $f(x_1,\ldots,x_n)=g(x_{\sigma(1)},\ldots,x_{\sigma(n)})$.
\end{defn}

It is well-known that equivalent functions share many properties. For example, two equivalent Boolean nested canalizing functions have the same sensitivity and the same average sensitivity \cite{Yua3,Yua1}. Having an equivalence relation on the set of all NCF's, the number of different equivalence classes of NCF's is of interest. However, we first need the following combinatorial result.

\begin{lemma}\label{lm4.1}\cite{Cha}(Page 70)
Given $n$, $r$ and $s_i$, $i=1,\ldots,r$ and $s=s_1+\cdots+s_r\leq n$. 
Then the number of integer solution of the equation $k_1+\cdots+k_r=n$, where $k_i\geq s_i$, is
\[\sum_{\substack{k_{1}+\cdots+k_{r}=n\\k_{i}\geq s_i, s=s_1+\cdots +s_r\leq n}}1 = \binom{r+n-s-1}{r-1}.\]
\end{lemma}

\begin{theorem}\label{th4.2}
The number of different equivalence classes of NCF's, defined as in Definition \ref{def4.1}, is
\[N=2^{n-1}(q-1)q^{n}(2^{q-1}-1)^n.\]
\end{theorem}
\begin{proof}

By only considering the number of different NCF's with a fixed canalizing variable order $\sigma$ in Equation \ref{eq4.1}, we get exactly the number of different equivalent classes. Thus, we can follow the same enumerative schedule as we did in the proof of Corollary \ref{co4.1}. The only difference is that we do not consider the permutation of the variables now. Hence, we get

\begin{align*}
N&=2^{n-1}q(q-2)(2^{q-1}-1)^n \sum_{\substack{r=2}}^{n}(q-1)^{r-1}\sum_{\substack{k_{1}+\cdots+k_{r-1}=n-1\\k_{i}\geq1,i=1,\ldots,r-1 }}1\\
&\qquad +2^{n}q(2^{q-1}-1)^n\sum_{\substack{r=1}}^{n-1}(q-1)^{r}\sum_{\substack{k_{1}+\cdots+k_{r}=n\\k_{i}\geq1,i=1,\ldots,r-1,k_{r}\geq2}}1\\
&=2^{n-1}q(q-2)(2^{q-1}-1)^n \sum_{\substack{r=2}}^{n}(q-1)^{r-1}\binom{n-2}{r-2}+2^{n}q(2^{q-1}-1)^n\sum_{\substack{r=1}}^{n-1}(q-1)^{r}\binom{n-2}{r-1}\\
&=2^{n-1}q(q-2)(2^{q-1}-1)^n (q-1)q^{n-2}+2^{n}q(2^{q-1}-1)^n(q-1)q^{n-2}\\
&=2^{n-1}(q-1)q^{n}(2^{q-1}-1)^n,
\end{align*}
where we used Lemma \ref{lm4.1} to eradicate the sums in the second equality.
\end{proof}

%%%%%%%%%%%%%%%%%%%%%%%%%%%%

\section{Discussion}\label{5}
In this paper, we generalized the concept of \emph{nested canalization} to finite fields without any restriction on the canalizing input sets. By using the product of indicator functions, a generalization of extended monomials, we successfully generalized the characterization of Boolean NCF's \cite{Yua1} to the multistate case. Besides, we discussed the permutational equivalence of NCF's, and a number of different equivalence classes was obtained. The main contributions of the paper are threefold. Firstly, we have established a closed form formula for the number of
multistage nested canalizing functions as defined in \cite{Mur}, improving on the recursive formula given there. Secondly, we
have established a very general definition of multistate nested canalizing function for which several analytical results still hold true. 
And here as well, we establish a closed form formula for the number of such functions.  And, thirdly, we have characterized and
calculated the number of equivalence classes of NCFs under permutation equivalence.

%In \cite{Yua2}, the authors generalized the results of \cite{Win2} to any finite set  without algebraic structures. One very interesting question for the future is whether the cardinality of NCF's can also be obtained in this setting.

\end{document}